\documentclass[11pt]{article} 
\usepackage{amsmath,amssymb,amsthm}  

\allowdisplaybreaks

\begin{document}

\def\fl#1{\left\lfloor#1\right\rfloor}
\def\cl#1{\left\lceil#1\right\rceil}
\def\ang#1{\left\langle#1\right\rangle}
\def\stf#1#2{\left[#1\atop#2\right]} 
\def\sts#1#2{\left\{#1\atop#2\right\}}
\def\eul#1#2{\left\langle#1\atop#2\right\rangle}
\def\N{\mathbb N}
\def\Z{\mathbb Z}
\def\R{\mathbb R}
\def\C{\mathbb C}

\newtheorem{theorem}{Theorem}
\newtheorem{Prop}{Proposition}
\newtheorem{Cor}{Corollary}
\newtheorem{Lem}{Lemma}
\newtheorem{Def}{Definition}  

\newenvironment{Rem}{\begin{trivlist} \item[\hskip \labelsep{\it
Remark.}]\setlength{\parindent}{0pt}}{\end{trivlist}}
\makeatletter
\renewenvironment{proof}[1][\proofname]{\par
  \normalfont
  \topsep6\p@\@plus6\p@ \trivlist
  \item[\hskip\labelsep{\bfseries #1}\@addpunct{\bfseries.}]\ignorespaces
}{%
  \endtrivlist
}
\renewcommand{\proofname}{\it Proof}
\makeatother

\title{Weighted Sylvester sums on the Frobenius set in more variables
}

\author{
Takao Komatsu\, and Yuan Zhang\\
\small Department of Mathematical Sciences, School of Science\\[-0.8ex]
\small Zhejiang Sci-Tech University\\[-0.8ex]
\small Hangzhou 310018 China\\[-0.8ex]
\small \texttt{komatsu@zstu.edu.cn}\, and \texttt{zjlgzy51@126.com}
}

\date{
\small MR Subject Classifications: Primary 11D07; Secondary 05A15, 05A17, 05A19, 11B68, 11D04, 11P81 
}

\maketitle
 
\begin{abstract} 
Let $a_1,a_2,\dots,a_k$ be positive integers with $\gcd(a_1,a_2,\dots,a_k)=1$. Let ${\rm NR}={\rm NR}(a_1,a_2,\dots,a_k)$ denote the set of positive integers nonrepresentable in terms of $a_1,a_2,\dots,a_k$. The largest nonrepresentable integer $\max{\rm NR}$, the number of nonrepresentable positive integers $\sum_{n\in{\rm NR}}1$ and the sum of nonrepresentable positive integers $\sum_{n\in{\rm NR}}n$ have been widely studied for a long time as related to the famous Frobenius problem.   
In this paper by using Eulerian numbers, we give formulas for the weighted sum $\sum_{n\in{\rm NR}}\lambda^{n}n^\mu$, where $\mu$ is a nonnegative integer and $\lambda$ is a complex number. We also examine power sums of nonrepresentable numbers and some formulae for three variables. Several examples illustrate and support our results. 
\end{abstract}

\section{Introduction}  

Given positive integers $a_1,\dots,a_k$ with $\gcd(a_1,\dots,a_k)=1$, it is well-known that all sufficiently large $n$ can be represented as a nonnegative integer combination of $a_1,\dots,a_k$. 
The {\it Frobenius Problem} is to determine the largest positive integer that is NOT representable as a nonnegative integer combination of given positive integers that are coprime (see \cite{ra05} for general references). This number is denoted by $g(a_1,\dots,a_k)$ and often called Frobenius number. The Frobenius Problem has been also known as the {\it Coin Exchange Problem} (or Postage Stamp Problem / Chicken McNugget Problem), which has a long history and is one of the problems that has attracted many people as well as experts. 

Let $n(a_1,\dots,a_k)$ be the number of positive integers with no nonnegative integer representation by $a_1,\dots,a_k$. It is sometimes called Sylvester number.  

According to Sylvester, for positive integers $a$ and $b$ with $\gcd(a,b)=1$, we have  
\begin{align*}
g(a,b)&=(a-1)(b-1)-1\quad{\rm \cite{sy1884}}\,,\\
n(a,b)&=\frac{1}{2}(a-1)(b-1)\quad{\rm \cite{sy1882}}\,. 
\end{align*}

There are many kinds of problems related to the Frobenius problem. The problems for the number of solutions (e.g., \cite{tr00}) and the sum of integer powers of values the gaps in numerical semigroups (e.g., \cite{bs93,fr07,fks}) are some of the interesting ones. 
One of other famous problems is about the so-called {\it Sylvester sums} 
$$
s(a_1,\dots,a_k)=\sum_{n\in{\rm NR}(a_1,\dots,a_k)}n 
$$ 
(see, e.g., \cite[\S 5.5]{ra05}, \cite{tu06} and references therein), where ${\rm NR}(a_1,\dots,a_k)$ denotes the set of positive integers without nonnegative integer representation by $a_1,\dots,a_k$. In addition, we denote by ${\rm R}(a_1,\dots,a_k)$ the set of positive integers with nonnegative integers representation by $a_1,\dots,a_k$. 
Brown and Shiue \cite{bs93} found the exact value of $s(a,b)$ for positive integers $a$ and $b$ with $\gcd(a,b)=1$,  
\begin{equation}
s(a,b)=\frac{1}{12}(a-1)(b-1)(2 a b-a-b-1)\,. 
\label{brown}
\end{equation} 
R\o dseth \cite{ro94} generalized Brown and Shiue's result by giving a closed form for 
$$
s_\mu(a,b):=\sum_{n\in{\rm NR}(a,b)}n^\mu\,,   
$$ 
where $\mu$ is a positive integer.  

When $k=2$, there exist beautiful closed forms for Frobenius numbers, Sylvester numbers and Sylvester sums, but 
when $k\ge 3$, exact determination of these numbers is difficult.  
The Frobenius number cannot be given by closed formulas of a certain type (Curtis \cite{cu90}). The problem to determine $g(a_1,\dots,a_k)$ is NP-hard under Turing reduction (see, e.g., Ram\'irez Alfons\'in \cite{ra05}). Nevertheless, the Frobenius number for some special cases are calculated (e.g., \cite{op08,ro56,se77}). One convenient formula is by Johnson \cite{jo60}. One analytic approach to the Frobenius number can be seen in \cite{bgk01,ko03}.    

Obtaining closed forms for general case are hopeless for $k\ge 3$. Several formulae for Frobenius numbers, Sylvester numbers and Sylvester sums have been considered under special cases.  For example, some formulas for the Frobenius number in three variables can be seen in \cite{tr17}. 

In fact, by introducing the Ap\'ery set below, it is easy to obtain the formulas for $g(A)$, $n(A)$ and $s(A)$. Without loss of generality, we assume that $a_1=\min(A)$. 

\begin{Def}  
For positive integers $A=\{a_1,a_2,\dots,a_k\}$ with $\gcd(A)=1$ and $a_1=\min(A)$ we denote by 
$$
{\rm Ape}(A)={\rm Ape}(a_1,a_2,\dots,a_k)=\{m_0,m_1,\dots,m_{a_1-1}\}\,, 
$$ 
the Ap\'ery set of $A$, where $m_i$ is the least positive integer of $R(A)$ satisfying $m_i\equiv i\pmod{a_1}$ $(1\le i\le a_1-1)$. 
Note that $m_0$ is defined to be $0$.  
\label{apery} 
\end{Def}

\begin{Lem}  
We have 
\begin{align*}
g(A)&=\left(\max_{1\le i\le a_1-1}m_i\right)-a_1\,,\quad{\rm \cite{bs62}}\\ 
n(A)&=\frac{1}{a_1}\sum_{i=1}^{a_1-1}m_i-\frac{a_1-1}{2}\,,\quad{\rm \cite{se77}}\\ 
s(A)&=\frac{1}{2 a_1}\sum_{i=1}^{a_1-1}m_i^2-\frac{1}{2}\sum_{i=1}^{a_1-1}m_i+\frac{a_1^2-1}{12}\,.\quad{\rm \cite{tr08}}
\end{align*}
\label{lem1} 
\end{Lem} 
Note that the third formula appeared with a typo in \cite{tr08}, and it has been corrected in \cite{pu18,tr17b}. 
\bigskip 

In this paper, we are interesting in giving a formula for the weighted sum 
$$
s_\mu^{(\lambda)}(A):=\sum_{n\in{\rm NR}(A)}\lambda^{n}n^\mu\,,
$$  
where $\lambda\in\mathbb C$. More precisely, if $\lambda\ne 1$ and $\lambda^{a_1}\ne1$, then  
\begin{align*}  
s_\mu^{(\lambda)}(A)&=\sum_{n=0}^\mu\frac{(-a_1)^n}{(\lambda^{a_1}-1)^{n+1}}\binom{\mu}{n}\sum_{j=0}^n\eul{n}{n-j}\lambda^{j a_1}\sum_{i=0}^{a_1-1}m_i^{\mu-n}\lambda^{m_i}\\
&\quad +\frac{(-1)^{\mu+1}}{(\lambda-1)^{\mu+1}}\sum_{j=0}^\mu\eul{\mu}{\mu-j}\lambda^j 
\end{align*}
(see Theorem \ref{th-hh} below).  Here, the Eulerian number $\eul{n}{l}$ is the number of permutations of $1,2,\dots,n$ in which exactly $l$ elements are greater than the previous element (see, e.g., \cite[\S 6.2]{gkp89},\cite[A008292]{oeis}). 
If $\lambda=1$, then 
$$ 
s_\mu^{(1)}(A)=\sum_{\kappa=0}^\mu\sum_{j=1}^{\kappa+1}\binom{\mu}{\kappa}\binom{\kappa+1}{j}\frac{(-1)^{j-1}}{\kappa+1}a_1^{\kappa-j}B_{\kappa-j+1}\sum_{i=1}^{a_1-1}(m_i-i)^j m_i^{\mu-\kappa}\,,  
$$ 
where $B_n$ is the Bernoulli number (see Theorem \ref{th-h1} below).

When $k=2$, the general formula for weighted sums is given in \cite{KZ0}.  In fact, the weighted power sums can be expressed in terms of the Apostol-Bernoulli numbers.  
When $k>2$, no general closed form for the functions $g$, $n$ and $s$ have been discovered. However, for some special cases or conditional cases, convenient forms have been known. Whenever an additional condition $a|{\rm lcm}(b,c)$ holds, the following results are known in \cite{tr17b}. 

\begin{Lem}
Set $l_1:={\rm lcm}(a,b)$ and $l_2:={\rm lcm}(a,c)$. Then, 
\begin{align*}  
g(a,b,c)&=l_1+l_2-(a+b+c)\,,\\
n(a,b,c)&=\frac{1}{2}\bigl(l_1+l_2-(a+b+c)+1\bigr)\,,\\
s(a,b,c)&=\frac{1}{12}\bigl(a^2+b^2+c^2+3(a b c+a b+b c+c a)\\
&\qquad -3(a+b+c)(l_1+l_2)+2(l_1^2+l_2^2)-1\bigr)\,. 
\end{align*}
\label{lem2}
\end{Lem} 

With the same condition $a|{\rm lcm}(b,c)$, we obtain the explicit formula 
\begin{align*}
&s^{(\lambda)}(a,b,c)\\
&=\frac{l_1(\lambda^{l_2}-1)+l_2(\lambda^{l_1}-1)+(l_1+l_2-a-b-c)(\lambda^{l_1}-1)(\lambda^{l_2}-1)}{(\lambda^a-1)(\lambda^b-1)(\lambda^c-1)}\\
&\quad -\frac{(\lambda^{l_1}-1)(\lambda^{l_2}-1)}{(\lambda^a-1)(\lambda^b-1)(\lambda^c-1)}\left(\frac{a}{\lambda^a-1}+\frac{b}{\lambda^b-1}+\frac{c}{\lambda^c-1}\right)\\ 
&\quad +\frac{\lambda}{(\lambda-1)^2}\,,   
\end{align*}
where  $\lambda\ne 1$ with $\lambda^{a}\ne1$, $\lambda^{b}\ne1$ and $\lambda^{c}\ne1$ (see Theorem \ref{th3} below).

\section{Main results}  

In this section, we consider the power sums of nonrepresentable numbers 
$$
s_\mu^{(\lambda)}(A):=\sum_{n\in{\rm NR}(A)}\lambda^n n^\mu\,,   
$$ 
where $\mu$ is a positive integer, and $A:=\{a_1,\dots,a_k\}$ is the set of positive integers with $\gcd(a_1,\dots,a_k)=1$.   

First, assume that $\lambda^{a_1}\ne 1$.

Eulerian numbers $\eul{n}{m}$ appear in the generating function
$$
\sum_{k=0}^\infty k^n x^k=\frac{1}{(1-x)^{n+1}}\sum_{m=0}^{n-1}\eul{n}{m}x^{m+1}\quad(n\ge 1)
$$ 
with $0^0=1$ and $\eul{0}{0}=1$ (\cite[p.244]{com74}), 
and have an explicit formula 
$$
\eul{n}{m}=\sum_{k=0}^{m+1}(-1)^k\binom{n+1}{k}(m-k+1)^n
$$   
(\cite[p.243]{com74}).  From the definition, we have 
$$
\eul{n}{k}=\eul{n}{n-k-1}\quad(0\le k\le n-1),
$$
in particular, we have
$$
\eul{n}{0}=\eul{n}{n-1}=1\quad\hbox{and}\quad\eul{n}{n}=0\quad(n\ge 1)\,.
$$

For $1\le i\le a_1-1$, let $m_i$ be defined in Definition \ref{apery}. Then, $m_i-a_1$ is the largest positive integer in ${\rm NR}(A)$ among those congruent to $i$ modulo $a_1$. Hence, there exists a nonnegative integer $\ell_i$ such that 
$$
m_i-a_1, m_i-2 a_1,\dots, m_i-\ell_i a_1\in{\rm NR}(A)
$$ 
with $m_i-\ell_i a_1>0$ and $m_i-(\ell_i+1)a_1<0$.  
We want to obtain 
\begin{equation}  
s_\mu^{(\lambda)}(A)=\sum_{n\in{\rm NR}(A)}\lambda^n n^\mu
=\sum_{i=0}^{a_1-1}\sum_{j=1}^{\ell_i}\lambda^{m_i-j a_1}(m_i-j a_1)^\mu\,. 
\label{eq:sml}
\end{equation}
Since $m_i\equiv i\pmod{a_1}$ and $1\le i\le a_1-1$, we get $\ell_i=(m_i-i)/a_1$. 
Let $\lambda\ne 1$ and $\lambda^{a_1}\ne1$. 
Then, the weighted sum of elements in ${\rm NR}(a_1,a_2,\dots,a_k)$ congruent to $i$ modulo $a_1$ is given by 
$$
\sum_{j=1}^{\ell_i}\lambda^{m_i-j a_1}(m_i-j a_1)^\mu\,. 
$$ 
For $n\ge 1$, we obtain 
\begin{align*} 
\sum_{j=1}^\infty\lambda^{m_i-j a_1}j^n&=\frac{\lambda^{m_i}}{(1-\lambda^{-a_1})^{n+1}}\sum_{h=0}^{n-1}\eul{n}{h}\lambda^{-(h+1)a_1}\\
&=\frac{\lambda^{m_i}}{(\lambda^{a_1}-1)^{n+1}}\sum_{h=0}^{n-1}\eul{n}{h}\lambda^{(n-h)a_1}\\
&=\frac{\lambda^{m_i}}{(\lambda^{a_1}-1)^{n+1}}\sum_{h=1}^{n}\eul{n}{n-h}\lambda^{h a_1}\,.  
\end{align*}
So, for $n\ge 0$, we get 
$$
\sum_{j=1}^\infty\lambda^{m_i-j a_1}j^n=\frac{\lambda^{m_i}}{(\lambda^{a_1}-1)^{n+1}}\sum_{h=0}^{n}\eul{n}{n-h}\lambda^{h a_1}\,. 
$$ 
Hence, for $\ell_i=(m_i-i)/a_1$
\begin{align*}
s_\mu^{(\lambda)}(A)&=\sum_{i=0}^{a_1-1}\sum_{j=1}^{\ell_i}\lambda^{m_i-j a_1}(m_i-j a_1)^\mu\\
&=\sum_{i=0}^{a_1-1}\sum_{j=1}^{\ell_i}\lambda^{m_i-j a_1}\sum_{n=0}^\mu\binom{\mu}{n}m_i^{\mu-n}j^n(-a_1)^n\\ 
&=\sum_{i=0}^{a_1-1}\sum_{n=0}^\mu\binom{\mu}{n}m_i^{\mu-n}(-a_1)^n\sum_{j=1}^\infty\lambda^{m_i-j a_1}j^n\\
&\quad -\sum_{i=0}^{a_1-1}\sum_{n=0}^\mu\binom{\mu}{n}m_i^{\mu-n}(-a_1)^n\sum_{j=\ell_i+1}^\infty\lambda^{m_i-j a_1}j^n\\
&=\sum_{i=0}^{a_1-1}\sum_{n=1}^\mu\binom{\mu}{n}m_i^{\mu-n}(-a_1)^n\sum_{j=1}^\infty\lambda^{m_i-j a_1}j^n\\
&\quad +\sum_{i=0}^{a_1-1}m_i^\mu\sum_{j=1}^\infty\lambda^{m_i-j a_1}\\
&\quad -\sum_{i=0}^{a_1-1}\sum_{n=0}^\mu\binom{\mu}{n}m_i^{\mu-n}(-a_1)^n\sum_{j=1}^\infty\lambda^{i-j a_1}\left(j+\frac{m_i-i}{a_1}\right)^n\\ 
&=\sum_{i=0}^{a_1-1}\sum_{n=1}^\mu\binom{\mu}{n}m_i^{\mu-n}(-a_1)^n\frac{\lambda^{m_i}}{(\lambda^{a_1}-1)^{n+1}}\sum_{h=1}^{n}\eul{n}{n-h}\lambda^{h a_1}\\
&\quad +\sum_{i=0}^{a_1-1}m_i^\mu\frac{\lambda^{m_i}}{\lambda^{a_1}-1}\\
&\quad -\sum_{i=0}^{a_1-1}\sum_{n=0}^\mu\binom{\mu}{n}m_i^{\mu-n}(-a_1)^n\sum_{j=1}^\infty\lambda^{i-j a_1}\left(j+\frac{m_i-i}{a_1}\right)^n\\ 
&=\sum_{n=0}^\mu\frac{(-a_1)^n}{(\lambda^{a_1}-1)^{n+1}}\binom{\mu}{n}\sum_{h=0}^n\eul{n}{n-h}\lambda^{h a_1}\sum_{i=0}^{a_1-1}m_i^{\mu-n}\lambda^{m_i}\\
&\quad -\sum_{i=0}^{a_1-1}\sum_{j=1}^\infty\lambda^{i-j a_1}(i-j a_1)^\mu\,. 
\end{align*}
Since the last term is equal to 
\begin{align*}
-\sum_{k=0}^\infty\lambda^{-k}(-k)^\mu&=\frac{(-1)^{\mu+1}}{(1-\lambda^{-1})^{\mu+1}}\sum_{m=0}^{\mu-1}\eul{\mu}{m}\lambda^{-(m+1)}\\ 
&=\frac{(-1)^{\mu+1}}{(\lambda-1)^{\mu+1}}\sum_{j=0}^\mu\eul{\mu}{\mu-j}\lambda^j\,, 
\end{align*}
we have the desired result.  

\begin{theorem}
For $1\le i\le a_1-1$, let $m_i$ be in the Ap\'ery set in Definition \ref{apery}. 
Assume that $\lambda\ne 1$ and $\lambda^{a_1}\ne 1$. Then for a positive integer $\mu$,  
\begin{align*}  
s_\mu^{(\lambda)}(A)&=\sum_{n=0}^\mu\frac{(-a_1)^n}{(\lambda^{a_1}-1)^{n+1}}\binom{\mu}{n}\sum_{j=0}^n\eul{n}{n-j}\lambda^{j a_1}\sum_{i=0}^{a_1-1}m_i^{\mu-n}\lambda^{m_i}\\
&\quad +\frac{(-1)^{\mu+1}}{(\lambda-1)^{\mu+1}}\sum_{j=0}^\mu\eul{\mu}{\mu-j}\lambda^j\,.
\end{align*}
\label{th-hh}
\end{theorem}

Using $\mu=2$ in the formula of Theorem \ref{th-hh}, we obtain the following.   

\begin{theorem}  
If $\lambda\ne 1$ and $\lambda^{a_1}\ne 1$, then 
\begin{align*}  
&s_2^{(\lambda)}(A)\\
&=\frac{1}{\lambda^{a_1}-1}\sum_{i=0}^{a_1-1}m_i^2\lambda^{m_i}-\frac{2 a_1\lambda^{a_1}}{(\lambda^{a_1}-1)^2}\sum_{i=0}^{a_1-1}m_i\lambda^{m_i}+\frac{a_1^2\lambda^{a_1}(\lambda^{a_1}+1)}{(\lambda^{a_1}-1)^3}\sum_{i=0}^{a_1-1}\lambda^{m_i}\\
&\quad -\frac{\lambda(\lambda+1)}{(\lambda-1)^3}\,.
\end{align*}
\label{th-h2}
\end{theorem}

When $\mu=1$ in the formula of Theorem \ref{th-hh}, we obtain the following.

\begin{theorem} 
If $\lambda\ne 1$ and $\lambda^{a_1}\ne1$, then  
\begin{align*}  
&s^{(\lambda)}(A):=s_1^{(\lambda)}(A)\\
&=\frac{1}{\lambda^{a_1}-1}\sum_{i=0}^{a_1-1}m_i\lambda^{m_i}
-\frac{a_1\lambda^{a_1}}{(\lambda^{a_1}-1)^2}\sum_{i=0}^{a_1-1}\lambda^{m_i}+\frac{\lambda}{(\lambda-1)^2}\,.
\end{align*} 
\label{th1}
\end{theorem}

If $\lambda\ne 1$ and $\lambda^{a_1}=\lambda^{a_2}=\cdots=\lambda^{a_k}=1$, then $\gcd(A)\ne 1$. 
Nevertheless, if $\lambda\ne 1$ and $\lambda^{a_1}=1$, then we have the following.   

\begin{theorem} 
If $\lambda\ne 0,1$ and $\lambda^{a_1}=1$, then  
$$
s^{(\lambda)}(A)=\frac{1}{2 a_1}\sum_{i=0}^{a_1-1}m_i^2\lambda^i-\frac{1}{2}\sum_{i=0}^{a_1-1}m_i\lambda^i+\frac{\lambda}{(\lambda-1)^2}\,.
$$ 
\label{th1b}
\end{theorem}  
\begin{proof} 
Since $\lambda^{a_1}=1$, the weighted sum of elements in ${\rm R}(A)$ congruent to $i$ modulo $a_1$ is given by 
\begin{align*}
&\sum_{j=1}^{\ell_i}\lambda^{m_i-j a_1}(m_i-j a_1)=m_i\sum_{j=1}^{\ell_i}\lambda^i-a_1\sum_{j=1}^{\ell_i}\lambda^i j\\
&=m_i\ell_i\lambda^i-a_1\frac{\ell_i(\ell_i+1)}{2}\lambda^i\\
&=\frac{m_i(m_i-i)\lambda^i}{a_1}-\frac{(m_i-i)^2\lambda^i}{2 a_1}-\frac{(m_i-i)\lambda^i}{2}\,. 
\end{align*}
Therefore, 
\begin{align*}  
&s^{(\lambda)}(A)\\
&=\frac{1}{a_1}\left(\sum_{i=0}^{a_1-1}m_i^2\lambda^i-\sum_{i=0}^{a_1-1}i m_i\lambda^i\right)\\
&\quad -\frac{1}{2 a_1}\left(\sum_{i=0}^{a_1-1}m_i^2\lambda^i-2\sum_{i=0}^{a_1-1}i m_i\lambda^i+\sum_{i=0}^{a_1-1}i^2\lambda^i\right)\\
&\quad -\frac{1}{2}\left(\sum_{i=0}^{a_1-1}m_i\lambda^i-\sum_{i=0}^{a_1-1}i\lambda^i\right)\\
&=\frac{1}{2 a_1}\sum_{i=0}^{a_1-1}m_i^2\lambda^i-\frac{1}{2}\sum_{i=0}^{a_1-1}m_i\lambda^i\\
&\quad -\frac{1}{2 a_1}\frac{a_1^2(\lambda-1)-2 a_1\lambda}{(\lambda-1)^2}+\frac{1}{2}\frac{a_1}{\lambda-1}\\ 
&=\frac{1}{2 a_1}\sum_{i=0}^{a_1-1}m_i^2\lambda^i-\frac{1}{2}\sum_{i=0}^{a_1-1}m_i\lambda^i+\frac{\lambda}{(\lambda-1)^2}\,. \qquad\qquad\qquad{\atop\qed} 
\end{align*}
\end{proof}
\bigskip 

In particular, when $\lambda=-1$ and $a_1$ is odd in Theorem \ref{th1}, 
we have the formula for alternate sums. When $k=2$ and $A=\{a,b\}$, the formulas are obtained in terms of Bernoulli or Euler numbers in \cite{wang08}. 

\begin{Cor}  
When $a_1$ is odd, we have 
$$
\sum_{n\in{\rm NR}(A)}(-1)^n n=-\frac{1}{2}\sum_{i=0}^{a_1-1}(-1)^{m_i}m_i+\frac{a_1}{4}\sum_{i=0}^{a_1-1}(-1)^{m_i}+\frac{a_1-1}{4}\,.
$$ 
\label{cor1}
\end{Cor}

When $k=2$ in Theorem \ref{th1}, since we have $\{m_i|1\le i\le a-1\}=\{b i|1\le i\le a-1\}$, 
if $\lambda^b\ne 1$ and $\lambda\ne 1$, then we have 
\begin{align*}  
&s^{(\lambda)}(a,b)\\
&=\frac{1}{\lambda^{a}-1}\sum_{i=0}^{a-1} b i\lambda^{b i}
-\frac{a\lambda^{a}}{(\lambda^{a}-1)^2}\left(1+\sum_{i=0}^{a-1}\lambda^{b i}\right)+\frac{\lambda}{(\lambda-1)^2}\\
&=\frac{1}{\lambda^{a}-1}\frac{b\bigl((a-1)\lambda^{a b+b}-a\lambda^{a b}+\lambda^{b}\bigr)}{(\lambda^b-1)^2}
-\frac{a\lambda^{a}}{(\lambda^{a}-1)^2}\left(1+\frac{\lambda^{a b}-\lambda^b}{\lambda^b-1}\right)\\
&\quad +\frac{\lambda}{(\lambda-1)^2}\\
&=\frac{\lambda}{(\lambda-1)^2}+\frac{a b \lambda^{a b}}{(\lambda^a-1)(\lambda^b-1)}-\frac{(\lambda^{a b}-1)\bigl((a+b)\lambda^{a+b}-a \lambda^a-b \lambda^b\bigr)}{(\lambda^a-1)^2(\lambda^b-1)^2}\,.
\end{align*}

If $\lambda\ne 1$ and $\lambda^b=1$, we have 
\begin{align}  
&s^{(\lambda)}(a,b)\notag\\
&=\frac{1}{\lambda^{a}-1}\sum_{i=0}^{a-1} b i
-\frac{a\lambda^{a}}{(\lambda^{a}-1)^2}\left(1+\sum_{i=0}^{a-1}1\right)+\frac{\lambda}{(\lambda-1)^2}\notag\\
&=\frac{b}{\lambda^{a}-1}\frac{a(a-1)}{2}-\frac{a^2\lambda^{a}}{(\lambda^{a}-1)^2}+\frac{\lambda}{(\lambda-1)^2}\notag\\
&=\frac{\lambda}{(\lambda-1)^2}+\frac{(a-1)a b}{2(\lambda^a-1)}-\frac{a^2\lambda^a}{(\lambda^a-1)^2}\,.
\label{eq:201}
\end{align}

\subsection{Examples}  

Consider the case $\mu=2$, $k=4$ and $(a_1,a_2,a_3,a_4)=(5,17,19,23)$.  
Since 
\begin{align*}  
&{\rm R}(5,17,19,23)\\
&=\{5,10,15,17,19,20,22,23,24,25,27,28,29,30,32,33,34,35,36,37,\dots\}\,,\\
&{\rm NR}(5,17,19,23)=\{1,2,3,4,6,7,8,9,11,12,13,14,16,18,21,26,31\}\,,
\end{align*}
we see that $m_1=36$, $m_2=17$, $m_3=23$ and $m_4=19$. By Theorem \ref{th-h2}, 
\begin{align*}
&s_2^{(\lambda)}(5,17,19,23)\\
&=\frac{1}{\lambda^{5}-1}(36^2\lambda^{36}+17^2\lambda^{17}+23^2\lambda^{23}+19^2\lambda^{19})\\
&\quad -\frac{10\lambda^{5}}{(\lambda^{5}-1)^2}(36\lambda^{36}+17\lambda^{17}+23\lambda^{23}+19\lambda^{19})\\
&\quad +\frac{25\lambda^{5}(\lambda^{5}+1)}{(\lambda^{5}-1)^3}(1+\lambda^{36}+\lambda^{17}+\lambda^{23}+\lambda^{19})\\
&\quad -\frac{\lambda(\lambda+1)}{(\lambda-1)^3}\,.
\end{align*}
The alternating sum, which is a typical weighted sum studied in \cite{wang08}, is the case $\lambda=-1$. The weight $\lambda$ does not have to be limited to real numbers. 
For example, we have 
\begin{align*}
&s_2^{(-1)}(5,17,19,23)=(-1)^1\cdot 1^2+(-1)^2\cdot 2^2+(-1)^3\cdot 3^2+(-1)^4\cdot 4^2\\
&\quad +(-1)^6\cdot 6^2+(-1)^7\cdot 7^2+(-1)^8\cdot 8^2+(-1)^9\cdot 9^2\\
&\quad +(-1)^{11}\cdot 11^2+(-1)^{12}\cdot 12^2+(-1)^{13}\cdot 13^2+(-1)^{14}\cdot 14^2\\
&\quad +(-1)^{16}\cdot 16^2+(-1)^{18}\cdot 18^2+(-1)^{21}\cdot 21^2+(-1)^{26}\cdot 26^2\\
&\quad +(-1)^{31}\cdot 31^2\\
&=-116\,,\\
&s_2^{(2)}(5,17,19,23)=2110129433818\,,\\
&s_2^{(\omega)}(5,17,19,23)=(-443+391\sqrt{-3})/2\,,
\end{align*}
where $\omega=(-1+\sqrt{-3})/2$. 
\bigskip 

Let $\mu=1$, $k=3$ and $(a_1,a_2,a_3)=(3,11,17)$.  Since  
\begin{align*}
{\rm R}(3,11,17)&=\{0,3,6,9,11,12,14,15,17,18,20,21,\dots\}\,,\\
{\rm NR}(3,11,17)&=\{1,2,4,5,7,8,10,13,16,19\}\,,
\end{align*} 
we have $m_1=22$ and $m_2=11$. Thus, elements of ${\rm NR}(3,11,17)$ which are congruent to $1$ resp. $2$ are $\{1,4,7,10,13,16,19\}$ and $\{2,5,8\}$. 
Our results also hold for negative numbers other than $-1$. 
When the weight is $\lambda=-2$, the weighted sum is equal to 
\begin{align*}
&s^{(-2)}(3,11,17)\\
&=(-2)^1\cdot 1+(-2)^2\cdot 2+(-2)^4\cdot 4+(-2)^5\cdot 5+(-2)^7\cdot 7+(-2)^8\cdot 8\\
&\quad +(-2)^{10}\cdot 10+(-2)^{13}\cdot 13+(-2)^{16}\cdot 16+(-2)^{19}\cdot 19\\ 
&=-9008090\,. 
\end{align*} 
On the other hand, from the equation of $s^{(\lambda)}=s_1^{(\lambda)}$ given in (\ref{eq:sml}), it can be calculated as 
\begin{align*}
&s^{(-2)}(3,11,17)\\
&=\bigl((-2)^1\cdot 1+(-2)^4\cdot 4+(-2)^7\cdot 7+(-2)^{10}\cdot 10+(-2)^{13}\cdot 13+(-2)^{16}\cdot 16\\
&\quad +(-2)^{19}\cdot 19\bigr)
+\bigl((-2)^{2}\cdot 2+(-2)^{5}\cdot 5+(-2)^{8}\cdot 8\bigr)\\ 
&=-9008090\,. 
\end{align*} 
By Theorem \ref{th1} directly, from $m_1=22$ and $m_2=11$, we have 
\begin{align*}
&s^{(-2)}(3,11,17)\\
&=\frac{1}{(-2)^{3}-1}\sum_{i=0}^{2}m_i(-2)^{m_i}
-\frac{3(-2)^{3}}{((-2)^{3}-1)^2}\left(1+\sum_{i=0}^{2}(-2)^{m_i}\right)+\frac{-2}{(-3)^2}\\
&=-\frac{1}{9}\bigl(22(-2)^{22}+11(-2)^{11}\bigr)+\frac{8}{27}\bigl(1+(-2)^{22}+(-2)^{11}\bigr)-\frac{2}{9}\\
&=-9008090\,. 
\end{align*}

\section{Power sums of nonrepresentable numbers}

Next, we give a summation formula for $\lambda=1$. We recall that Bernoulli numbers $B_n$ are defined by 
$$
\frac{x}{e^x-1}=\sum_{n=0}^\infty B_n\frac{x^n}{n!}\,.
$$  

\begin{theorem}  
For $1\le i\le a_1-1$, let $m_i$ be defined in Definition \ref{apery}. 
$$ 
s_\mu^{(1)}(A)=\sum_{\kappa=0}^\mu\sum_{j=1}^{\kappa+1}\binom{\mu}{\kappa}\binom{\kappa+1}{j}\frac{(-1)^{j-1}}{\kappa+1}a_1^{\kappa-j}B_{\kappa-j+1}\sum_{i=0}^{a_1-1}(m_i-i)^j m_i^{\mu-\kappa}\,.  
$$ 
\label{th-h1}
\end{theorem}

Using $\mu=1$ in Theorem \ref{th-h1}, we obtain the third formula in Lemma \ref{lem1}. 
Using $\mu=1$ and $k=2$, that is, $m_i=b i$ ($1\le i\le a-1$), we retrieve the formula (\ref{brown}).

\begin{proof}[{\it Proof of Theorem \ref{th-h1}}.]   
First observe 
$$
\sum_{j=1}^{\ell_i} j^\kappa=\frac{1}{\kappa+1}\sum_{h=1}^{\kappa+1}\binom{\kappa+1}{h}(-1)^{\kappa-h+1}B_{\kappa-h+1}\ell_i^h  
$$ 
(see, e.g., \cite[(1.1)]{AIK}\footnote{In this book, Bernoulli numbers $\mathbb B_n$ are defined by $x/(1-e^{-x})=\sum_{n=0}^\infty\mathbb B_n\frac{x^n}{n!}$, satisfying $\mathbb B_n=(-1)^n B_n$ ($n\ge 0$)}). 
Since $\ell_i=(m_i-i)/a_1$, the power sum of elements in ${\rm R}(A)$ congruent to $i$ modulo $a_1$ is 
\begin{align*}  
&\sum_{j=1}^{\ell_i}(m_i-j a_1)^\mu\\
&=\sum_{j=1}^{\ell_i}\sum_{\kappa=0}^\mu\binom{\mu}{\kappa}m_i^{\mu-\kappa}(-j a_1)^\kappa\\ 
&=\sum_{\kappa=0}^\mu\binom{\mu}{\kappa}m_i^{\mu-\kappa}(-1)^\kappa a_1^\kappa\sum_{j=1}^{\ell_i} j^\kappa\\ 
&=\sum_{\kappa=0}^\mu\binom{\mu}{\kappa}m_i^{\mu-\kappa}(-1)^\kappa a_1^\kappa\frac{1}{\kappa+1}\sum_{h=1}^{\kappa+1}\binom{\kappa+1}{h}(-1)^{\kappa-h+1}B_{\kappa-h+1}\ell_i^h\\
&=\sum_{\kappa=0}^\mu\sum_{h=1}^{\kappa+1}\binom{\mu}{\kappa}\binom{\kappa+1}{h}\frac{(-1)^{h-1}}{\kappa+1}a_1^{\kappa-h}B_{\kappa-h+1}(m_i-i)^h m_i^{\mu-\kappa}\,. 
\end{align*}
Hence, we have 
\begin{align*}  
&s_\mu^{(1)}(A)\\
&=\sum_{i=0}^{a_1-1}\sum_{j=1}^{\ell_i}(m_i-j a_1)^\mu\\
&=\sum_{\kappa=0}^\mu\sum_{j=1}^{\kappa+1}\binom{\mu}{\kappa}\binom{\kappa+1}{j}\frac{(-1)^{j-1}}{\kappa+1}a_1^{\kappa-j}B_{\kappa-j+1}\sum_{i=0}^{a_1-1}(m_i-i)^j m_i^{\mu-\kappa}\,. \qquad{\atop\qed}   
\end{align*}
\end{proof}

\subsection{Example}   

For $k=3$ and $(a_1,a_2,a_3)=(3,11,17)$, we have $m_1=22$ and $m_2=11$ and hence 
\begin{align*}  
&s_\mu^{(1)}(3,11,17)\\
&=\sum_{\kappa=0}^\mu\sum_{j=1}^{\kappa+1}\binom{\mu}{\kappa}\binom{\kappa+1}{j}\frac{(-1)^{j-1}}{\kappa+1}3^{\kappa-j}B_{\kappa-j+1} 
(21^j\cdot 22^{\mu-\kappa}+9^j\cdot 11^{\mu-\kappa})\,.
\end{align*}
From the equation of $s_\mu^{(\lambda)}$ given in (\ref{eq:sml}), we have 
$$
s_\mu^{(1)}(3,11,17)=
1^\mu+2^\mu+4^\mu+5^\mu+7^\mu+8^\mu+10^\mu+13^\mu+16^\mu+19^\mu\,.
$$

\section{Three variables} 

When $k>2$, there is no closed form for general $g(A)$, $n(A)$ and $s(A)$. But for some specific cases, some closed forms are known. In \cite{tr17}, closed forms for $g(a,b,c)$, $n(a,b,c)$ and $s(a,b,c)$ are given under the condition $a|{\rm lcm}(b,c)$.  
In this section, we give a closed form for weighted sums 
$$
s^{(\lambda)}(a,b,c):=\sum_{n\in{\rm NR}(a,b,c)}\lambda^n n
$$ 
under an additional condition $a|{\rm lcm}(b,c)$.

For convenience, put $r=\gcd(a,b)$ and $s=\gcd(a,c)$. Then $b s={\rm lcm}(a,b):=l_1$ and $c r={\rm lcm}(a,c):=l_2$.

\begin{theorem}  
Assume that positive integers $a,b,c$ satisfy $\gcd(a,b,c)=1$ and $a|{\rm lcm}(b,c)$.  
If $\lambda\ne 1$ with $\lambda^{a}\ne1$, $\lambda^{b}\ne1$ and $\lambda^{c}\ne1$, then 
\begin{align*}
&s^{(\lambda)}(a,b,c)\\
&=\frac{l_1(\lambda^{l_2}-1)+l_2(\lambda^{l_1}-1)+(l_1+l_2-a-b-c)(\lambda^{l_1}-1)(\lambda^{l_2}-1)}{(\lambda^a-1)(\lambda^b-1)(\lambda^c-1)}\\
&\quad -\frac{(\lambda^{l_1}-1)(\lambda^{l_2}-1)}{(\lambda^a-1)(\lambda^b-1)(\lambda^c-1)}\left(\frac{a}{\lambda^a-1}+\frac{b}{\lambda^b-1}+\frac{c}{\lambda^c-1}\right)\\ 
&\quad +\frac{\lambda}{(\lambda-1)^2}\,. 
\end{align*}
\label{th3}
\end{theorem}
\begin{proof}
We use Theorem \ref{th1} with $k=3$. First observe 
$$
\{m_i|0\le i\le a-1\}=\{b x+c y|0\le x\le s-1,\,0\le y\le r-1\}  
$$ 
(\cite[Theorem 8]{tr17b}).  
We have 
\begin{align*}
&s^{(\lambda)}(a,b,c)\\
&=\frac{1}{\lambda^{a}-1}\sum_{i=0}^{a-1}m_i\lambda^{m_i}
-\frac{a\lambda^{a}}{(\lambda^{a}-1)^2}\sum_{i=0}^{a-1}\lambda^{m_i}+\frac{\lambda}{(\lambda-1)^2}\\
&=\frac{1}{\lambda^{a}-1}\sum_{x=0}^{s-1}\sum_{y=0}^{r-1}(b x+c y)\lambda^{b x+c y}
-\frac{a\lambda^{a}}{(\lambda^{a}-1)^2}\sum_{x=0}^{s-1}\sum_{y=0}^{r-1}\lambda^{b x+c y}\\
&\quad +\frac{\lambda}{(\lambda-1)^2}\\
&=\frac{1}{\lambda^{a}-1}\left(b\sum_{x=0}^{s-1}x\lambda^{b x}\sum_{y=0}^{r-1}\lambda^{c y}+c\sum_{x=0}^{s-1}\lambda^{b x}\sum_{y=0}^{r-1}y\lambda^{c y}\right)\\
&\quad-\frac{a\lambda^{a}}{(\lambda^{a}-1)^2}\sum_{x=0}^{s-1}\lambda^{b x}\sum_{y=0}^{r-1}\lambda^{c y}
 +\frac{\lambda}{(\lambda-1)^2}\\
&=\frac{b}{\lambda^{a}-1}\left(\frac{(s-1)\lambda^{b s}}{\lambda^b-1}-\frac{\lambda^{b s}-\lambda^b}{(\lambda^b-1)^2}\right)\frac{\lambda^{c r}-1}{\lambda^c-1}\\
&\quad +\frac{c}{\lambda^{a}-1}\frac{\lambda^{b s}-1}{\lambda^b-1}\left(\frac{(r-1)\lambda^{c r}}{\lambda^c-1}-\frac{\lambda^{c r}-\lambda^c}{(\lambda^c-1)^2}\right)\\
&\quad-\frac{a\lambda^{a}}{(\lambda^{a}-1)^2}\frac{\lambda^{b s}-1}{\lambda^b-1}\frac{\lambda^{c r}-1}{\lambda^c-1}
 +\frac{\lambda}{(\lambda-1)^2}\\
&=\frac{b s(\lambda^{c r}-1)+c r(\lambda^{b s}-1)+\bigl(b(s-1)+c(r-1)-a\bigr)(\lambda^{b s}-1)(\lambda^{c r}-1)}{(\lambda^a-1)(\lambda^b-1)(\lambda^c-1)}\\
&\quad -\frac{(\lambda^{b s}-1)(\lambda^{c r}-1)}{(\lambda^a-1)(\lambda^b-1)(\lambda^c-1)}\left(\frac{a}{\lambda^a-1}+\frac{b}{\lambda^b-1}+\frac{c}{\lambda^c-1}\right) +\frac{\lambda}{(\lambda-1)^2}\,. {\atop\qed} 
\end{align*}
\end{proof}  
\bigskip

Because of the conditions $\gcd(a,b,c)=1$ and $a|{\rm lcm}(b,c)$, we cannot have $\zeta^a\ne 1$ and $\zeta^b=\zeta^c=1$.  
Assume that $\zeta^a\ne 1$, $\zeta^b\ne 1$ and $\zeta^c=1$. Then by Theorem \ref{th1}, we get  
\begin{align*}  
&s^{(\lambda)}(a,b,c)\\
&=\frac{1}{\lambda^{a}-1}\sum_{i=0}^{a-1}m_i\lambda^{m_i}
-\frac{a\lambda^{a}}{(\lambda^{a}-1)^2}\sum_{i=0}^{a-1}\lambda^{m_i}+\frac{\lambda}{(\lambda-1)^2}\\
&=\frac{1}{\lambda^{a}-1}\left(b\sum_{x=0}^{s-1}x\lambda^{b x}\cdot r+c\sum_{x=0}^{s-1}\lambda^{b x}\cdot\frac{r(r-1)}{2}\right)\\
&\quad -\frac{a\lambda^{a}}{(\lambda^{a}-1)^2}\sum_{x=0}^{s-1}\lambda^{b x}\cdot r+\frac{\lambda}{(\lambda-1)^2}\\
&=\frac{b r}{\lambda^a-1}\left(\frac{(s-1)\lambda^{b s}}{\lambda^b-1}-\frac{\lambda^{b s}-\lambda^b}{(\lambda^b-1)^2}\right)+\frac{c}{\lambda^a-1}\frac{\lambda^{b s}-1}{\lambda^b-1}\frac{r(r-1)}{2}\\
&\quad-\frac{a r\lambda^a}{(\lambda^a-1)^2}\frac{\lambda^{b s}-1}{\lambda^b-1}+\frac{\lambda}{(\lambda-1)^2}\\
&=\frac{r(\lambda^{b s}-1)}{(\lambda^a-1)(\lambda^b-1)}\left(b(s-1)+\frac{c(r-1)}{2}-a-\frac{a}{\lambda^a-1}-\frac{b}{\lambda^b-1}\right)\\
&\quad +\frac{b r s}{(\lambda^a-1)(\lambda^b-1)}+\frac{\lambda}{(\lambda-1)^2}\,.
\end{align*}
Hence, we obtain the following.

\begin{theorem}  
Assume that positive integers $a,b,c$ satisfy $\gcd(a,b,c)=1$ and $a|{\rm lcm}(b,c)$.  
If $\lambda\ne 1$ with $\lambda^{a}\ne1$, $\lambda^{b}\ne1$ and $\lambda^{c}=1$, then 
\begin{align*}
&s^{(\lambda)}(a,b,c)\\
&=\frac{l_2(\lambda^{l_1}-1)}{c(\lambda^a-1)(\lambda^b-1)}\left(l_1+\frac{l_2}{2}-a-b-\frac{c}{2}-\frac{a}{\lambda^a-1}-\frac{b}{\lambda^b-1}\right)\\ 
&\quad +\frac{l_1 l_2}{c(\lambda^a-1)(\lambda^b-1)}+\frac{\lambda}{(\lambda-1)^2}\,. 
\end{align*}
\label{th4} 
\end{theorem}

\section{Acknowledgement} 

The authors would like to thank the referee for the useful and detailed comments and suggestions, which made this paper more attractive and more sophisticated.

\end{document}